\documentclass{article}
\usepackage{amssymb}
\usepackage{amsmath}

\newtheorem{theorem}{Theorem}

\newtheorem{definition}[theorem]{Definition}

\newtheorem{lemma}[theorem]{Lemma}

\newenvironment{proof}[1][Proof]{\noindent\textbf{#1.} }{\ \rule{0.5em}{0.5em}}
\begin{document}
\title{Orthogonal direct sum of Lie groups}

\author {Babak Hassanzadeh\\
Department of mathematics,\\
Azarbaijan university of shahid madani,\\
 Tabriz, Iran.\\
E-mail: babakmath777@gmail.com}

\maketitle


\renewcommand{\thefootnote}{}


\renewcommand{\thefootnote}{\arabic{footnote}}
\setcounter{footnote}{0}
\begin{abstract}
In the present paper, we study orthogonal direct sum of Lie groups with a contact form by defining matrix Lie groups and its Lie algebra. By this short preface we show the main idea and we will bring some definitions and concepts in every stage.

\end{abstract}

\section{Introduction}
Lie groups as smooth manifolds are excellent structures for considering geometric properties because we can use more algebraic tools for our study. Specialy Nilpotent Lie groups play an important role in areas of mathematics, and 2-step nilpotent Lie groups have a special significance. They are the nonabelian contact Lie groups that came as close as possible to being abelian, but the admit interesting phenomena that do not arise in abelian groups. In this paper we study differential geometry of Lie groups by a Riemannian left invariant metric $ \langle , \rangle $. We also give a matching theory interpretation to the current result on Riemannian left invariant metrics see \cite{co}, \cite{cn} on the algebraic analogue of matchings.  \\
One would expect to find some properties that are similar to those in flat\\ Euclidean space, which in this context one may regard as a simply connected, abelian Lie group of translations with a canonical left invariant metric. Such\\ properties do exist, but other geometric properties of Lie groups are foreign to Euclidean geometry.\\
 All of these topics are important in mathematics and physics. We will study the Hyper Kahler geometry of matrix Lie algebras while Lie group G is contact one.

\section{Preliminaries}
In this topic we will study the properties of Lie algebra of the Lie group. In the next sections we will study going to be more specialized and focused. Now we start to the basic concepts. First we say a Lie group H of a Lir group G is a subgroup which is also a submanifold. 
\begin{definition}
 Here $ F=\mathcal{R} $ or $ \mathcal{C} $. A Lie algebra over F is pair $ (\mathfrak{g},[.,.]) $, where $ \mathfrak{g} $ is a vector space over F and
\begin{align*}
[.,.]:\mathfrak{g}\times \mathfrak{g} \rightarrow \mathfrak{g},
\end{align*}
is an F-bilinear map satisfying the following properties
\begin{align*}
[X,Y]=-[Y,X], \quad
[X,[Y,Z]]+[Z,[X,Y]]+[Y,[Z,X]]=0,
\end{align*}
\end{definition} 
for all $X,Y \in\mathfrak{g} $.The latter is the Jacobi identity. And we have
\begin{align*}
[X,Y]=\nabla_{X}Y-\nabla_{Y}X.
\end{align*}
A Lie subalgebra of a Lie algebra is a vector space that is closed under the bracket. 
Following theorem recognize Lie algebra of a Lie group.
\begin{theorem}
Let G be a Lie group and $\mathfrak{g}$ is a set of left invariant vector fields on G, then we have \newline
(1)$\mathfrak{g}$ is a vector space and map $ E:\mathfrak{g} \rightarrow T_{e}G  $ is defined by $ X \mapsto X_{e} $ is a linear isomorphism and therefore $ dim\mathfrak{g}=dim T_{e}G=dim G $. (e is identity element)\newline
(2) Left invariant vector fields necessity are differentiable.\newline
(3)$(\mathfrak{g},[.,.])$ is Lie algebra.
\end{theorem}
\begin{definition}
Lie algebra made of left invariant vector field on Lie group G is called Lie algebra's of Lie group G, this Lie algebra is isomorph with $T_{e}G$, and we have
\begin{align*}
[X_{e},Y_{e}]=[X,Y]_{e},
\end{align*}
where $X$ and $Y$ are arbitrary left invariant vector fields.
\end{definition}
Following theorem show the relationship between Lie subgroups and Lie algebras.
\begin{theorem}
 Let G be a Lie group.\\
(a) If H is a Lie subgroup of G, then $ \mathfrak{h} \simeq T_{e}H \subset T_{e}G \simeq \mathfrak{g} $ g is a Lie subalgebra.\\
(b) If $ \mathfrak{h} \subset \mathfrak{g} $ a Lie subalgebra, there exists a unique connected Lie subgroup $ H \subset G $ with Lie algebra $ \mathfrak{h} $.
\end{theorem}
Let G be a Lie group equipted with Riemannian connection, by using only these identities and combining a few permutations of variables obtain the formula 
\begin{align}
g(\nabla_{X}Y,Z) &=\dfrac{1}{2}\lbrace Xg(Y,Z)+Yg(X,Z)- Zg(X,Y)\\
                        &-g([X,Y],Z)+g([Z,X],Y)+g([Z,Y],X)\rbrace \nonumber.
\end{align}
\begin{definition}
Let G be a connected Lie group. The subgroup
\begin{align*}
Z(G)=\lbrace x \in G:xy=yx, \forall y \in G \rbrace ,
\end{align*}
is called the center of G It is a Lie subgroup with corresponding Lie subalgebra
\begin{align*}
Z(\mathfrak{g}) =\lbrace X \in \mathfrak{g} : [X,Y]=0, \forall Y\in \mathfrak{g} \rbrace.
\end{align*}
\end{definition}
A Lie subalgebra $ \mathfrak{h}\subset \mathfrak{g} $ is called an ideal if $ [\mathfrak{h},\mathfrak{g}]\subset \mathfrak{h} $. The Lie algebra of a normal Lie subgroup of G is necessarily an ideal. A Lie algebra $ \mathfrak{g} $ is called simple if it has no nontrivial ideals (that is {0} and $ \mathfrak{g} $ are the only ideals in $ \mathfrak{g} $). It is called semi-simple if it is a direct sum of simple Lie algebras or contains no nonzero solvable ideals. Note that this in particular implies that the center $ Z(\mathfrak{g})=0 $. A Lie group is called simple (respectively semi-simple) if its Lie algebra is simple (respectively semi-simple).\\
Next theorem express useful information about semi-simple Lie algebra.
\begin{theorem}$ [7] $
 Let $ \mathfrak{g} $ be a semi-simple Lie algebra.\newline
(i) If I is ideal of $ \mathfrak{g} $, then $ \mathfrak{g}=I \oplus I^{\perp} $.\newline
(ii) If $ \mathfrak{g} $ be a semi-simple Lie algebra, then any subideal of  $ \mathfrak{g} $ is ideal of  $ \mathfrak{g} $.\newline
(iii) If  $ \mathfrak{g} $ is semi simple, any ideal of  $ \mathfrak{g} $ is semi simple.
\end{theorem}
\begin{theorem}$ [7] $
 Let g be a left invariant Riemannian metric on a connected Lie group G. This metric will also be right invariant if and only if ad(X) is skew-adjoint for every $ X \in  \mathfrak{g} $.
\end{theorem}
\begin{definition}
A nilpotent Lie group is a Lie group G which is connected and whose Lie algebra is nilpotent Lie algebra $ \mathfrak{g} $, that is, it's Lie algebra has a sequence of ideals of $ \mathfrak{g} $ by $ \mathfrak{g}^{0}=\mathfrak{g} $, $ \mathfrak{g}^{1}=[\mathfrak{g},\mathfrak{g}] $,$ \mathfrak{g}^{2}=[\mathfrak{g},\mathfrak{g}^{1}] $,..., $ \mathfrak{g}^{i}=[\mathfrak{g},\mathfrak{g}^{i-1} ]  $. Also $ \mathfrak{g} $ is called nilpotent if $ \mathfrak{g}^{n}=0 $ for some n.
\end{definition}
Let $ \mathfrak{h} $ be a Lie subalgebra of $ \mathfrak{g} $, the subalgebra    
\begin{align*}
N_{\mathfrak{g}}(\mathfrak{h})=\left\lbrace X \in \mathfrak{g} \mid [X,\mathfrak{h}] \subseteq \mathfrak{h} \right\rbrace,
\end{align*}
is called normalizer of $ \mathfrak{h} $ in $ \mathfrak{g} $.
\begin{definition}
A  Lie subalgebra $ \mathfrak{h} $ of $ \mathfrak{g} $ is called Cartan subalgebra and shown by $ \mathfrak{h} \preceq_{csa} \mathfrak{g} $, if\\
1)$ \mathfrak{h} $ is nilpotent.\\
2)$ \mathfrak{h} $ is self normalized i. e. $ N_{\mathfrak{g}}(\mathfrak{h})=\mathfrak{h} $.
\end{definition}
We say $ \mathfrak{h} $ is a self-normalizing subalgebra  of $ \mathfrak{g} $ so that 
\begin{align*}
N_{\mathfrak{g}}(\mathfrak{h})=\mathfrak{h}
\end{align*}
Throughout this paper $ \mathfrak{h} $ is 2-step nilpotent. Let z be center of $\mathfrak{h}$ and $ z^{\perp} $ is the normal complement of z in $ \mathfrak{h} $, so that $ \mathfrak{h} =z+z^{\perp} $, we will study $ \mathfrak{h} $ by a skew symmetric map $ j(Z):z^{\perp}\rightarrow z^{\perp} $ define for every element of z. j is given by $ j(Z)X=(adX)^{*}(Z) $ for all $ X \in z^{\perp} $, ad X is the adjoint of ad X relative to the inner product $ \langle , \rangle $. Finally $ j(Z) $ is defined by 
\begin{align}
\langle j(Z)X,Y \rangle =\langle [X,Y],Z \rangle
\end{align}
For all $ X,Y \in z^{\perp} $. Let $ V \in \mathfrak{g} $ be an arbitrary element, if $ [V,Z] \in \mathfrak{g} $ then $ V \in \mathfrak{h} $ and $ [V,Z]=0 $. 
In the other case if $ [V,X] \in \mathfrak{h} $, then $ V \in \mathfrak{h} $ and $ [V,X] \in z $ or $ [V,Z]=0 $.$[1]$\\
Also, from (2.1) we obtain 
\begin{align*}
\langle [V,X],Y \rangle =\langle [V,Y],X \rangle =0
\end{align*}
Then the Gauss and Weingarten formulas of H in G are given respectively by
\begin{align}
\bar{\nabla}_{X}Y=\nabla_{X}Y+h(X,Y)
\end{align}
\begin{align}
\bar{\nabla}_{X}\xi =-A_{\xi}X+D_{X}\xi
\end{align}
for any vector fields X,Y tangent to H and any vector field $ \xi $ normal to H, where h is the second fundamental form, D the normal covariant derivative and A is the Weingarten map the submanifold H in G. The relation between second fundamental form and Weingarten map S is the 
\begin{align}
\langle A_{\xi}X,Y \rangle =\langle h(X,Y),\xi \rangle .
\end{align}
Now by simple calculation we have 
\begin{align}
\bar{\nabla}_{X}Y+\bar{\nabla}_{Y}X=2h(X,Y)
\end{align}
whenever $ X,Y \in z^{\perp} $. For the second fundamental form h, we define the covariant derivative $ \bar{\nabla}h $ of h with respect to the covariant derivative by
\begin{align}
\bar{\nabla}_{X}h(Y,Z)=D_{X}(h(Y,Z))-h(\nabla_{X}Y,Z)-h(Y,\nabla_{X}Z).
\end{align}
and the mean curvature vector of G is defined by
\begin{align}
H=\left(\dfrac{1}{n} \right)trace  h \quad n=dim G
\end{align}
The equation of Gauss is given by
\begin{align}
\tilde{R}(X,Y;Z,W)=R(X,Y;Z,W)+\langle h(X, Z), h(Y,W)\rangle - \langle h(X,W), h(Y, Z)\rangle,
\end{align}
for X,Y,Z,W tangent to G, where R and $ \tilde{R} $ denote the curvature tensors of G and H, respectively. Finally we bring up the short about contact structure. 
Let $(M,\varphi, \xi, \eta, g)$ be an almost contact manifold, i.e. M is a $(2n+1)$-dimensional differentiable manifold with a left invariant almost contact structure $(\varphi, \xi, \eta)$ consisting of an endomorphism $ \phi $ of the tangent bundle, a vector field $ \xi $, its dual 1-form $  \eta$ as well as M is equipped with a Riemannian metric $g$, such that the following algebraic relations are satisfied
\begin{align*}
\varphi \xi=0, \quad \phi^{2}=-Id+\eta\otimes \xi, \quad \eta \xi=1,\\
g(\varphi X,\varphi Y)=g(X,Y)-\eta(X)\eta(Y),
\end{align*}
where Id is the identity and X,Y are elements of the tangent bundle TM of the smooth vector fields on M. Let $ \Phi $ denote the 2-form in M given by $ \Phi(X,Y)=g(X,\varphi Y) $. The 2-form $ \Phi $ is called the fundamental 2-form in M and the manifold is said to be a contact metric manifold if $ \Phi=d\eta $. If $ \xi $ is a Killing vector field with respect to g, the contact metric structure is called a K-contact structure. It is easy to prove that a contact metric manifold is K-contact if and only if $ \nabla_{X}\xi=-\varphi X $, for any $ X \in TM $. Throuthout of this paper almost contact structure is left invariant.
The following definition is from \cite{cm}.
\begin{definition}
A group G is said to be fail having the acyclic matching at order $ m \in \mathbb{N} \cup \lbrace \infty\rbrace $, if there exist subsets A ,B of G at matchings $ f,g:A\rightarrow B $ such that $ f\neq g $ and $ m_{f}=m_{g} $.
\end{definition}
\section{Orthogonal direct sum of Lie groups}
In this section we consider subgroups of a Lie group in a new array with a left invariant Riemannian metric. 
Let G be a Lie group and $ \mathfrak{g} $ is Lie algebra of G and it's equipped with Riemannian left invariant metric, so we show it by g(,). Let $ X,Y,Z,W,V,K,N,M \in \mathfrak{g} $ are arbitrary left invariant vector fields, we put them in a $2\times2$ matrices array like 
\begin{align}
\mathbf{A} =
\begin{bmatrix}
X & W \\
Z & Y  \\
\end{bmatrix}
,
\mathbf{B} = 
\begin{bmatrix}
V & K \\
N & M  \\
\end{bmatrix}
\end{align}
where $ \bar{\mathfrak{g}} $ is all of such matrixs and it's trivial $ \bar{\mathfrak{g}} $ is a vector space. Now we define an inner product on $ \bar{\mathfrak{g}} $ by g(,).
\begin{align}
\langle  
\begin{bmatrix}
X & W \\
Z & Y  \\
\end{bmatrix}, 
\begin{bmatrix}
V & K \\
N & M  \\
\end{bmatrix} \rangle =g(X,V)+g(W,K)+g(Z,N)+g(Y,M)
\end{align}
We may see $ \langle , \rangle $ is a left invariant metric. Suppose for all smooth function on $ \bar{\mathfrak{g}} $ like f, there are smooth functions a,b,c,d
  on $ \mathfrak{g} $, such that  
\begin{align*}
f=
\begin{bmatrix}
a & b \\
c & d  \\
\end{bmatrix}
\end{align*}
and we have
\begin{align*}
f(A)=
\begin{bmatrix}
a(X) & b(W) \\
c(Z) & d(Y)  \\
\end{bmatrix} 
\end{align*}
Now we can define covariant derivative on $ \bar{\mathfrak{g}} $.
 Let $ \bar{\nabla} $ denote the covariant derivative on $ \bar{\mathfrak{g}} $, then
\begin{align}
\bar{\nabla}_{
\begin{bmatrix}
X & W \\
Z & Y  \\
\end{bmatrix} }    
\begin{bmatrix}
V & K \\
N & M  \\
\end{bmatrix}=
\begin{bmatrix}
\nabla_{X}V & \nabla_{W}K \\
\nabla_{Z}N & \nabla_{Y}M  \\
\end{bmatrix} 
\end{align}
Therefore; we can define bracket for $ \bar{\mathfrak{g}} $\\
\begin{align}
\left[ 
\begin{bmatrix}
X & W \\
Z & Y  \\
\end{bmatrix},
\begin{bmatrix}
V & K \\
N & M  \\
\end{bmatrix}
\right]=
\begin{bmatrix}
[X,V] & [W,K] \\
[Z,N] & [Y,M]  \\
\end{bmatrix}
\end{align}
Using (13) show that $ \bar{\mathfrak{g}} $ is a Lie algebra and we say it is matrix Lie algebra of $ \mathfrak{g} $ and Lie group of it is called matrix Lie group. From (10) we  realize $ \langle A,B \rangle=0 $ if and only if all of entries in the same location are orthogonal. Also, if
\begin{align*}
\langle  
\begin{bmatrix}
X & 0 \\
0 & Y  \\
\end{bmatrix}, 
\begin{bmatrix}
V & K \\
N & M  \\
\end{bmatrix}\rangle =
\langle  
\begin{bmatrix}
X & W \\
Z & Y  \\
\end{bmatrix}, 
\begin{bmatrix}
V & K \\
N & M  \\
\end{bmatrix} \rangle 
\end{align*}
for fix elements of $\mathfrak{g} $, then 
\begin{align*}
g(W,K)=g(Z,M)=0,
\end{align*}
but if all of elements of matrix B are arbitrary, then $ W=Z=0 $.\\
Transpose matrix already, is one of the most recognized, for example 
\begin{equation*}
\mathbf{A^{t}} =
\begin{bmatrix}
X & Z \\
W & Y  \\
\end{bmatrix}
\end{equation*}
we can compute 
\begin{align}
\langle A,A \rangle=\Vert X \Vert^{2}+\Vert W \Vert^{2}+\Vert Z \Vert^{2}+\Vert Y \Vert^{2}
\end{align}
and
\begin{align}
\langle A,A^{t} \rangle =\Vert X \Vert^{2}+\Vert Y \Vert^{2}+2g(W,Z)
\end{align}
Therefore $ \langle A,A \rangle=\langle A,A^{t} \rangle $ if and only if 
\begin{align}
2g(W,Z)=\Vert W \Vert^{2}+\Vert Z \Vert^{2}.
\end{align}
 Let $ A,B \in \bar{\mathfrak{g}} $ are arbitrary elements such that $ \langle A,B \rangle=0 $, by straightforward computation we will have 
\begin{flushleft}
(1)$ \langle A^{t},B^{t} \rangle=0 $,\\
(2)$ \langle A,B^{t}\rangle =\langle A^{t},B \rangle $,\\
(3)$ \langle A,A^{t}\rangle =0 $ if main diagonal of A be zero and other diagonal be orthogonal.
\end{flushleft}
Let dim$ \mathfrak{g}=n $ and $ \lbrace e_{i} \rbrace $, $ i=1,...,n $, be an orthonormal base for $ \mathfrak{g} $, easily we can conclude 
\begin{align*}
\begin{bmatrix}
e_{i} & 0 \\
0 & 0  \\
\end{bmatrix},
\begin{bmatrix}
0 & e_{i} \\
0 & 0  \\
\end{bmatrix},
\begin{bmatrix}
0 & 0 \\
e_{i} & 0  \\
\end{bmatrix}, 
\begin{bmatrix}
0 & 0 \\
0 & e_{i}  \\
\end{bmatrix}, i=1,...,n 
\end{align*}
It's trivial for dim$ \mathfrak{g}=n $ dimension of $ \bar{\mathfrak{g}} $ is 4n.\\
\textbf{Curvature tensors.} If X,Y,Z  are left invariant vector fields on Lie group G, then $ R(X,Y)Z=\nabla_{X}\nabla_{Y}Z-\nabla_{Y}\nabla_{X}Z-\nabla_{[X,Y]}Z $ is also left invariant. Let F,E,P,S are arbitrary elements of $\mathfrak{g} $, we get
\begin{align*}
\bar{R}\left(
\begin{bmatrix}
X & W \\
Z & Y  \\
\end{bmatrix},
\begin{bmatrix}
V & K \\
N & M  \\
\end{bmatrix} \right)
\begin{bmatrix}
F & E \\
P & S  \\
\end{bmatrix} =
\begin{bmatrix}
R(X,V)F & R(W,K)F \\
R(Z,N)P & R(Y,M)S  \\
\end{bmatrix} 
\end{align*}
Let A,B be orthogonal and entries of them are unit, then sectional curvature is calculated in the following way
\begin{align*}
\bar{K}\left(
\begin{bmatrix}
X & W \\
Z & Y  \\
\end{bmatrix},
\begin{bmatrix}
V & K \\
N & M  \\
\end{bmatrix}\right)=
\begin{bmatrix}
K(X,V) & K(W,K) \\
K(Z,N) & K(Y,M)  \\
\end{bmatrix}
\end{align*}
 Now we define one of the most important concepts in the present paper. For any $ A \in \bar{\mathfrak{g}} $, we have 
\begin{equation}
O
\begin{bmatrix}
X & W \\
Z & Y  \\
\end{bmatrix} =g(X,Y)-g(Z,W)
\end{equation}
If $ O(A)=0 $, then
\begin{align*}
g(X,Y)=g(Z,W)
\end{align*}
therefore g(X,Y) is zero if and only if g(Z,W) be zero.
Let A and B are two diagonal matrix and  
\begin{align*}
\mathbf{A} = 
\begin{bmatrix}
X & 0 \\
0 & Y  \\
\end{bmatrix},
\mathbf{B} = 
\begin{bmatrix}
Z & 0 \\
0 & W  \\
\end{bmatrix} 
\end{align*}
From the definition of bracket we conclude that the set of all diagonal matrices are Lie subalgebra. It's easy to consider $ \langle A,A^{t} \rangle $ is nonzero. Also, we can see  $ \langle A,A^{*} \rangle=0 $ if and only if $ O(A)=0 $. 
If for an arbitrary element $ A \in \bar{g} $, there is a same entry in a row or column and  $ O(A)=0 $, then other elements are equal. For example if  
\begin{align*}
\mathbf{A} =
\begin{bmatrix}
X & V \\
Y & Y  \\
\end{bmatrix} 
\end{align*}
we have $ O(A)=g(X,Y)-g(Y,V) $ and it's trivial $ X=V $, for$ g(X,Y)\neq 0 $. Otherwise
 we will have $ g(Y,V)=0 $.\\
Using (1) and some calculation, we obtain
\begin{align}
 O
\begin{bmatrix}
\nabla_{X}Y & W \\
V & Z  \\
\end{bmatrix}=
 O
\begin{bmatrix}
\dfrac{1}{2}[X,Y] & W \\
V & Z  \\
\end{bmatrix}+ 
O
\begin{bmatrix}
\dfrac{1}{2}[Z,X] & X \\
\dfrac{1}{2}[Y,Z] & Y  \\
\end{bmatrix} .
\end{align}
From (17) we conclude:\\
If X,Y are commutative with Z, then
\begin{align}
 O
\begin{bmatrix}
\nabla_{X}Y & W \\
V & Z  \\
\end{bmatrix}=
 O
\begin{bmatrix}
\dfrac{1}{2}[X,Y] & W \\
V & Z  \\
\end{bmatrix}.
\end{align}
If $ [X,Y]=0 $ and V,W are orthogonal, then
\begin{align*}
 O
\begin{bmatrix}
\nabla_{X}Y & W \\
V & Z  \\
\end{bmatrix}=
O
\begin{bmatrix}
\dfrac{1}{2}[Z,X] & X \\
\dfrac{1}{2}[Y,Z] & Y  \\
\end{bmatrix}.
\end{align*}
Let we know X and Y completely, hence when we study $\mathfrak{g} $ we can recognize the relation between Z and W if we can put them in a matrix like A such that $ O(A)=0 $. In continuation of the article we will try to get results in $\mathfrak{g} $ by $ \bar{\mathfrak{g}} $ and vise versa. In this type elements of $ \bar{\mathfrak{g}} $, $ \langle A,A^{t} \rangle $ is only depended to arrays are in main diagonal, 
\begin{align*}
\langle 
\begin{bmatrix}
X & W \\
Z & Y  \\
\end{bmatrix}, 
\begin{bmatrix}
X & W \\
Z & Y  \\
\end{bmatrix}\rangle =g(X,X)+2g(X,Y)+g(Y,Y)
\end{align*}
In the other case we consider self adjoint elements of $ \bar{\mathfrak{g}} $,
let A be an arbitrary element of matrix Lie algebra presented in (3.1), so that $ O(A)=0 $ and it is a self adjoint element, then 
\begin{equation*}
\mathbf{A^{*}} =
\begin{bmatrix}
Y & -W \\
-Z & X  \\
\end{bmatrix}
\end{equation*}
Then, we have
\begin{align*}
\langle A,A^{*} \rangle =\langle  
\begin{bmatrix}
X & W \\
Z & Y  \\
\end{bmatrix}, 
\begin{bmatrix}
Y & -W \\
-Z &  X \\
\end{bmatrix} \rangle &=2g(X,Y)-g(Z,Z)-g(W,W)\\
                                     &=2g(Z,W)-g(Z,Z)-g(W,W).
\end{align*}
If A and $ A^{*} $ are orthogonal, we get
\begin{align}
2g(Z,W)=g(Z,Z)+g(W,W)
\end{align}
thus Z=W and $ A=A^{t} $. In the following by straightforward computation we get these facts, whereas $ A=A^{t} $.
\begin{align*}
&(1)\langle A^{t},A^{*^{t}} \rangle =2O(A),\\
&(2)\langle A^{t},A^{t^{*}} \rangle =\langle A,A^{*} \rangle.\\
&(3)\langle A,A \rangle =\langle A^{*},A^{*} \rangle =\langle A^{t},A^{t} \rangle.\\
\end{align*}
Several forms can be defined on matrix Lie algebra, which some of them have interesting properties, we introduce one of them. Let A and B are members, which presented at (10), we define
\begin{align*}
\left(A,B \right) &=det
\begin{bmatrix}
g(X,V) & g(W,K) \\
g(Z,N) & g(Y,M)  \\
\end{bmatrix} \\
                     &=g(X,V)g(Y,M)-g(W,K)g(Z,N). 
\end{align*}
Immediately we can see if $ \langle A,B \rangle=0 $, then $ (A,B)=0 $. Furthermore; if one 
row or column be zero, then $ (A,B)=0 $. By straight calculations we get
\begin{align*}
&(1)(A,B)=(B,A),\\
&(2)(A^{t},B)=(A,B^{t}),\\
&(3)(A^{*},B)=(A,B^{*}).\\
\end{align*}
In the other case suppose A is a diagonal matrix, then
\begin{align*}
(A,A)=(A,A^{t})=\vert X\vert^{2}\vert Y \vert^{2}
\end{align*}
and 
\begin{align*}
(A,A^{*})=g(X,Y)^{2}
\end{align*}
It's trivial if X,Y are orthogonal, then $ (A,A^{*})=0 $.\\
\begin{lemma} 
Let $ A \in \bar{\mathfrak{g}} $ be an arbitrary element presented in (10) such that $ O(A)=0 $ and $ (A,A)=0 $, then main diagonal elements are parallel.
\end{lemma}
\begin{proof}
\begin{align*}
(A,A)=\vert X\vert^{2}\vert Y \vert^{2}-g(W,Z)^{2}
\end{align*}
Then useing $ g(X,Y)=g(W,Z) $, we obtain 
\begin{align*}
g(X,Y)^{2}=\vert X\vert^{2}\vert Y \vert^{2}
\end{align*}
Therefore $ g(X,Y)=\vert X \vert \vert Y\vert $. From 
\begin{align*}
\cos \theta =\dfrac{g(X,Y)}{\vert X \vert \vert Y\vert }
\end{align*}
we conclude $ \theta =0 $ and X,Y are parallel vector fields.
\end{proof}\\
Following statements could be concluded directly from lemma 10. 
\begin{align*}
&(1)g(X,Y)\neq 0 \quad and \quad g(Z,W)\neq 0,\\
&(2)\langle A,A^{t} \rangle =(\vert X \vert +\vert Y\vert)^{2},\\
&(3)\langle A,A^{*}\rangle =2\vert X\vert \vert Y\vert -\vert W\vert^{2}-\vert Z\vert^{2},\\
&(4)(A,A^{t})=0,\\
&(5)(A,A^{*})=\vert X\vert^{2}\vert Y \vert^{2}-\vert Z\vert^{2}\vert W \vert^{2},\\
\end{align*}
\begin{theorem}
Let $ A \in \bar{\mathfrak{g}} $ be an arbitrary element presented in (10) such that $ \langle A,A^{*} \rangle=0 $ and $ (A,A^{*})=0 $. If $ O(A)=0 $, then W=Z. 
\end{theorem}
\begin{proof}
First we need some equations, then
\begin{align}
\langle 
\begin{bmatrix}
X & W \\
Z & Y  \\
\end{bmatrix},
\begin{bmatrix}
Y & -W \\
-Z &  X \\
\end{bmatrix} \rangle =2g(X,Y)-g(W,W)-g(Z,Z),
\end{align}
Now from  
\begin{align*}
(A,A^{*})=g(X,Y)^{2}-g(W,W)g(Z,Z),
\end{align*}
we get
\begin{align}
g(X,Y)^{2}=g(W,W)g(Z,Z).
\end{align}
Using (22) we obtain
\begin{align}
g(X,Y)^{2}=\dfrac{1}{4}\lbrace g(W,W)^{2}+g(Z,Z)^{2}+2g(Z,Z)g(W,W)\rbrace .
\end{align}
Comparing (22) and (23) show that
\begin{align}
g(W,W)^{2}+g(Z,Z)^{2}=2g(Z,Z)g(W,W),
\end{align}
and we have $ (\vert W \vert^{2} -\vert Z \vert^{2})^{2}=0 $ and 
$ \vert Z\vert =\vert W\vert $. Since $ O(A)=0 $ and from (21) we conclude
\begin{align*}
g(X,Y)=g(W,W)=g(Z,Z)=g(Z,W)
\end{align*}
and the proof is complete.
\end{proof}\\
As a corollary of theorem 11,  suppose $ (A,A)=(B,B) $ and $ O(A)=O(B)=0 $, thus 
\begin{align*}
\vert X \vert^{2}\vert Y \vert^{2}-\vert Z \vert^{2}\vert W \vert^{2}=\vert V \vert^{2} \vert S \vert^{2}-\vert N \vert^{2}\vert M \vert^{2}.
\end{align*}
Easily, we get
\begin{align}
\vert X \vert^{2}\vert Y \vert^{2}-\vert V \vert^{2} \vert S \vert^{2}=\vert Z \vert^{2}\vert W \vert^{2}-\vert N \vert^{2}\vert M \vert^{2},
\end{align}
we put 
\begin{align*}
\mathbf{A^{'}} =
\begin{bmatrix}
X & S \\
V & Y  \\
\end{bmatrix},
\mathbf{B^{'}} = 
\begin{bmatrix}
Z & M \\
N & W  \\
\end{bmatrix}
\end{align*}
Left hand of (25) is equal to $ O(A^{'}) $ and right hand is equal to $ O(B^{'}) $ and 
$ O(A^{'})=O(B^{'}) $.\\
Now we want to study some Lie subalgebras of $ \bar{\mathfrak{g}} $, for this we identify some elements of $ \bar{\mathfrak{g}} $, then we will give the definition of subalgebra.  
\begin{definition}
Every element of $ \bar{\mathfrak{g}} $ that entries in main diagonal and the other diagonal are orthogonal is called cross element. 
\end{definition}
In the other words for any $ A \in \bar{\mathfrak{g}} $,
\begin{align}
\mathbf{A} = 
\begin{bmatrix}
X & W \\
Z & Y  \\
\end{bmatrix}
\end{align}
we have $ g(X,Y)=g(Z,W)=0 $. It's trivial if A be a cross element, then $ A^{t} $ and $ A^{*} $ are cross, too.\\
Suppose $ \mathfrak{h} $ be a Lie subalgebra of $ \mathfrak{g} $ and $ \mathfrak{n} $ be orthogonal complement of it with Riemannian left invariant metric, such that $ \mathfrak{g}=\mathfrak{h}+\mathfrak{n} $. If entries of the first row of A be from $ \mathfrak{h} $ and second row's elements are chosen from $ \mathfrak{n} $, then it will be cross element and if B be an another element of  $ \bar{\mathfrak{g}} $ with same properties of A, then $ [A,B] $ is cross and entries of first row are in $ \mathfrak{h} $ and second row's entries of they are chosen from $ \mathfrak{n} $. Therefore set of all elements like A make a Lie subalgebra of $ \bar{\mathfrak{g}} $ and we called it subalgebra type one and show it by $ C^{1} $. Immediately we realize for every element of cross subalgebra like A, $ O(A)=0 $ and clearly $ C^{1} $ is cross.\\
{\bf Notation} If entries of first row of A be from $ \mathfrak{n} $ and second row's elements are chosen from $ \mathfrak{h} $, it will be type one subalgebra, but it's other subalgebras.\\
In the other case for A, let $ X,Y \in \mathfrak{h} $ and $ Z,W \in \mathfrak{n} $, therefore A is not a cross element, while entries are arbitrary left invariant vector fields. All elements like A make a Lie subalgebra of $ \bar{\mathfrak{g}} $ and we show it by $ C^{2} $ and say subalgebra type two.\\
Finally, we let $ X,Z \in \mathfrak{h} $ and $ W,Y \in \mathfrak{n} $ and these elements make a Lie subalgebra of $ \bar{\mathfrak{g}} $. We say it is third type subalgebra and it's cross one. We show it by $ C^{3} $.\\
With a quick review we find that.\\
(1)Let $ A \in C^{1} $, then $ A^{t},A^{*} \in C^{3} $.\\
(2)Let $ A \in C^{2} $, then $ A^{t},A^{*} \in C^{2} $.\\ 
(3)Let $ A \in C^{3} $, then $ A^{t},A^{*} \in C^{1} $.\\
Let $ A \in C^{1} $ and $ B \in \bar{\mathfrak{g}} $ so that first row's entries are selected from $ \mathfrak{n} $ and second row's is from $ \mathfrak{h} $. Then $ \langle A,B \rangle =0 $ and 
\begin{align*}
[A,B]=
\begin{bmatrix}
[\mathfrak{h},\mathfrak{n}] & [\mathfrak{h},\mathfrak{n}]  \\
[\mathfrak{n},\mathfrak{h}] & [\mathfrak{n},\mathfrak{h}]  \\
\end{bmatrix},
\end{align*}
therefore $ [A,B] $ cannot be an element of $ C^{1} $. We can bring such explanation  for $ C^{2} $ and $ C^{3} $. We must be noticed that if $ X,Z \in \mathfrak{n} $ and $ W,Y \in \mathfrak{h} $, then $ A \in C^{3} $, but it is different subalgebra. Also, for $ C^{1} $
and $ C^{2} $. Special if we define matrix Lie algebras for higher n, for example $ 3\times 3 $ or more.\\
Let M be a subamnifold of G contains the identity element and $ \mathfrak{m} $ be a Lie algebra of it and there are $ A,B \in \bar{\mathfrak{m}} $ such that entries of them are selected from $ \mathfrak{m} $. A, B are presented in (10). Then 
\begin{align*}
\bar{\nabla}_{A}B  &=
\begin{bmatrix}
\tilde{\nabla}_{X}V & \tilde{\nabla}_{W}K \\
\tilde{\nabla}_{Z}N & \tilde{\nabla}_{Y}M  \\
\end{bmatrix}=
\begin{bmatrix}
\nabla_{X}V+h(X,V) & \nabla_{W}M+h(W,K) \\
\nabla_{Z}N+h(Z,N) & \nabla_{Y}S+h(Y,M)  \\
\end{bmatrix} \\ 
                          &=\begin{bmatrix}
\nabla_{X}V & \nabla_{W}K \\
\nabla_{Z}N & \nabla_{Y}M  \\
\end{bmatrix}+
\begin{bmatrix}
h(X,V) & h(W,K) \\
h(Z,N) & h(Y,M)  \\
\end{bmatrix}.
\end{align*}
$ \bar{\nabla} $, $ \tilde{\nabla} $ and $ \nabla $ are covariant derivatives on $ \bar{\mathfrak{m}} $, $ \mathfrak{g} $ and $ \mathfrak{m} $. In this case we have 
\begin{align*}
\bar{h}(A,B)=
\begin{bmatrix}
h(X,V) & h(W,K) \\
h(Z,N) & h(Y,M)  \\
\end{bmatrix} \
\end{align*}
$ \bar{h} $ is second fundamental form of $ \bar{\mathfrak{h}} $, which it is induced subalgebra of $ \mathfrak{h} $. By straightforward calculations, we have 
\begin{align*}
\bar{h}(A^{t},B^{t})=\bar{h}^{t}(A,B).
\end{align*}
Ultimately, we consider the decompositions of $ \mathfrak{g} $ and $ \bar{\mathfrak{g}} $, then try to find the relations between them.
If $ \bar{\mathfrak{h}} $ be a Lie subalgebra of $ \bar{\mathfrak{g}} $ and $ \bar{\mathfrak{n}} $ be an orthogonal complement of it by $ \bar{g} $ such that 
$ \bar{\mathfrak{g}}=\bar{\mathfrak{h}}+\bar{\mathfrak{n}} $. This decomposition induced a decomposition on $ \mathfrak{g} $ that 
$ \mathfrak{g}=\mathfrak{h}+\mathfrak{n} $, so $ \bar{\mathfrak{h}} $ and $ \bar{\mathfrak{n}} $ are matrix Lie algebra of $ \mathfrak{h} $ and $ \mathfrak{n} $, respectively. Vice versa let $ \mathfrak{h} $ be aLie subalgebra of $ \mathfrak{g} $ and 
$ \mathfrak{n} $ is orthogonal complement of it by $g$. For both of $ \mathfrak{h} $ and $  \mathfrak{n} $ we can define their matrix Lie algebras, but we have some elements in  
$ \bar{\mathfrak{g}} $, there are not in $ \bar{\mathfrak{h}} $ and $ \bar{\mathfrak{n}} $, for example 
\begin{equation*}
\begin{bmatrix}
X & K \\
N & Y  \\
\end{bmatrix},
\begin{bmatrix}
X & W \\
N & Z  \\
\end{bmatrix}
\end{equation*}
for all $ X,Y,Z,W \in \mathfrak{h} $ and $ V,K,M,N \in \mathfrak{n} $. Therefore we get the following decomposition for matrix Lie algebra
\begin{align}
\bar{\mathfrak{g}}=\bar{\mathfrak{h}}+\bar{\mathfrak{n}}+\bar{\mathfrak{c}}
\end{align}
We continue by investigation of $ \bar{\mathfrak{c}} $. all of $ C^{1} $,$ C^{2} $ and $ C^{3} $ are in $ \bar{\mathfrak{c}} $.\\
Now we consider an element of $ \bar{\mathfrak{g}} $ which in $ C^{1} $,$ C^{2} $ and $ C^{3} $ with same entries in different places, then 
\begin{align*}
\mathbf{A_{1}} = 
\begin{bmatrix}
X & Y \\
Z & W \\
\end{bmatrix} ,
\mathbf{A_{2}} =
\begin{bmatrix}
X & W \\
Z & Y  \\
\end{bmatrix} ,
\mathbf{A_{3}} =
\begin{bmatrix}
X & W \\
Y & Z  \\
\end{bmatrix} 
\end{align*}
$ X,Y \in \mathfrak{h} $ and $ Z,W \in \mathfrak{n} $. It's trivial $ O(A_{1})=O(A_{3})=0 $. Next lemma show the one of the interesting relation in this type elements.
\begin{lemma} 
Let $ \langle A_{2},A^{*}_{3} \rangle =0 $ and $ O(A_{2})=0 $, then \\
(1)$ \langle A_{2},A_{3} \rangle =\langle A_{2},A^{t}_{3} \rangle $,\\
(2)$ \langle A_{1},A^{*}_{1} \rangle =\langle A_{3},A^{*}_{3} \rangle $,\\
(3)$ \langle A_{1},A^{t}_{1} \rangle =\langle A_{3},A^{t}_{3} \rangle $,
\end{lemma}
\begin{proof}
By hyphotesises we have $ g(X,Y)=g(W,W) $ and $ g(X,Y)=g(Z,W) $, now we can see Z=W, then 
\begin{align}
\langle A_{2},A_{3} \rangle =\vert X \vert^{2}+\vert W \vert^{2}=\vert X \vert^{2}+g(X,Y),
\end{align}
and
\begin{align}
\langle A_{2},A^{t}_{3} \rangle =\vert X \vert^{2}+g(Z,W)=\vert X \vert^{2}+g(X,Y),
\end{align} 
Comparing (28) and (29) prove the first episode. For (2) we have 
\begin{align*}
\langle A_{3},A^{*}_{3} \rangle =-\vert W \vert^{2}-\vert Y \vert^{2},\\
\langle A_{1},A^{*}_{1} \rangle =-\vert Y \vert^{2}-\vert Z \vert^{2}.
\end{align*}
Now (2) is trivial. Finally
\begin{align*}
\langle A_{3},A^{t}_{3} \rangle =+\vert X \vert^{2}+\vert Z \vert^{2},\\
\langle A_{1},A^{t}_{1} \rangle =+\vert X \vert^{2}+\vert W \vert^{2},
\end{align*}
the proof of (3) is distinctive.
\end{proof}

We end up this section with the following result from \cite{cp}, see also \cite{ck} for more details on matchings. 
\begin{theorem}
Algebraic number fields have tha maximal linear matching property.
\end{theorem}
\subsection{Almost complex structure}
Let J be a left invariant almost complex structure on G so that $ (G,g,J) $ is a Hermitian Lie group, then 
\begin{align*}
\langle 
\begin{bmatrix}
JX & JW \\
JZ & JY  \\
\end{bmatrix} ,
\begin{bmatrix}
JV & JK \\
JN & JM  \\
\end{bmatrix}  \rangle =\langle 
\begin{bmatrix}
X & W \\
Z & Y  \\
\end{bmatrix} ,
\begin{bmatrix}
V & K \\
N & M  \\
\end{bmatrix} \rangle
\end{align*}
and 
\begin{align*}
\langle 
\begin{bmatrix}
JX & JW \\
JZ & JY  \\
\end{bmatrix},
\begin{bmatrix}
X & W \\
Z & Y  \\
\end{bmatrix}\rangle =0.
\end{align*}
Therefore J induce almost complex structure $ \bar{J} $ on $ \bar{\mathfrak{g}} $ such that 
\begin{align*}
\bar{J}
\begin{bmatrix}
X & W \\
Z & Y  \\
\end{bmatrix}=
\begin{bmatrix}
JX & JW \\
JZ & JY  \\
\end{bmatrix} 
\end{align*}
Now it's trivial $ \bar{\mathfrak{g}} $ is Hermitian if and only if $ \mathfrak{g} $ be Hermitian. Such statement is valid for Kaehler Lie algebras.
\begin{lemma} 
If $ \mathfrak{g} $ be a Hermitian Lie algebra, then $ O(\bar{J}A)=O(A) $ for any $ A \in \bar{\mathfrak{g}} $.
\end{lemma}
Using lemma 14, we get
\begin{align*}
O
\begin{bmatrix}
JX & W \\
Z & Y  \\
\end{bmatrix} =O 
\begin{bmatrix}
X & W \\
Z & -JY  \\
\end{bmatrix} =O 
\begin{bmatrix}
Z & JY \\
X & W  \\
\end{bmatrix} ,
\end{align*}
\begin{align*}
O
\begin{bmatrix}
JX & JW \\
Z & Y  \\
\end{bmatrix}=O 
\begin{bmatrix}
X & W \\
-JZ & -JY  \\
\end{bmatrix}=O 
\begin{bmatrix}
JZ & JY \\
X & W  \\
\end{bmatrix},
\end{align*}
Also, $ O(A)=O(A^{t})=O(A^{*}) $ and 
\begin{align*}
&(1)\langle \bar{J}A,A^{t} \rangle =\langle \bar{J}A,A^{*} \rangle =\langle \bar{J}A^{t},A^{*} \rangle =0,\\
&(2)\bar{J}A^{t}=(\bar{J}A)^{t},\\
&(3)\bar{J}A^{*}=(\bar{J}A)^{*},\\
&(4)If A=-A^{t}, then (\bar{J}A^{t})=-(\bar{J}A)^{t},\\
&(5)If A=-A^{*}, then (\bar{J}A^{*})=-(\bar{J}A)^{*}.\\
\end{align*}
Nonzero element A of $ \bar{\mathfrak{g}} $ is called almost complex if one row or one column be by under the influence of almost complex structure, for example 
\begin{align*}
\mathbf{A_{C}} = 
\begin{bmatrix}
JX & W \\
JZ & Y  \\
\end{bmatrix} ,
\mathbf{A_{R}} =
\begin{bmatrix}
JX & JW \\
Z & Y  \\
\end{bmatrix},
\end{align*}
By some initial calculations, we obtain.
\begin{align*}
&(1)(A_{R})^{t}=(A^{t})_{C},\\
&(2)\langle A_{C},A_{C} \rangle =\langle A_{R},A_{R} \rangle =\langle A,A\rangle ,\\
&(3)\langle A_{C},A \rangle =\vert W \vert^{2}+\vert Y \vert^{2},\\
&(4)\langle A_{R},A \rangle =\vert Z \vert^{2}+\vert Y \vert^{2},\\
&(5)\langle A_{C},A_{R} \rangle =\vert X \vert^{2}+\vert Y \vert^{2},\\
&(6)O(A_{C})+O(A_{R})=2g(JX,Y),\\
&(7)If  O(A_{C})=O(A_{R}) , then  g(JZ,W)=0.\\
\end{align*}
Now we can state the following theorem.
\begin{theorem}
Let $ A \in \bar{\mathfrak{g}} $ be an arbitrary element, presented in (10), if $ (\bar{J}A^{t},A^{*})=0 $, then $ O(A_{C})=0 $.
\end{theorem}
\begin{proof}
\begin{align*}
(\bar{J}A^{t},A^{*})=-g(JX,Y)^{2}+g(JZ,W)^{2}
\end{align*}
then $ g(JX,Y)=g(JZ,W) $ and from 
\begin{align*}
O(\mathbf{A}_{C})= 
\begin{bmatrix}
JX & W \\
JZ & Y  \\
\end{bmatrix} =g(JX,Y)-g(JZ,W),
\end{align*}
proof is trivial.
\end{proof}\\
Let $ O(A_{R})=0 $, then $ g(JX,Y)=g(Z,JW) $ and we get 
\begin{align}
O(A_{C})=2g(JW,Z).
\end{align}
On the other hand if $ O(A_{C})=0 $, then $ g(JX,Y)=g(JZ,W) $ and we have 
\begin{align}
O(A_{R})=2g(JZ,W).
\end{align}
Now we can give the following lemma
\begin{lemma}
Let $ A \in \bar{\mathfrak{g}} $ be an arbitrary element, presented in (10), if $ O(A_{C})=O(A_{R}) $, then
\begin{align*}
g(JX,Y)=0.
\end{align*}
\end{lemma}
\begin{proof}
Useing (30) and (31) show desired result.
\end{proof}\\
Furthermore; if $ A_{C} $ be a cross element, then $ O(A_{R})=0 $ and if $ A_{R} $ be cross, then $ O(A_{C})=0 $.\\
Now let A be an element of $ C^{1} $ and $ \mathfrak{g}=\mathfrak{h}+\mathfrak{n} $
, in this case we can see that $ \mathfrak{h} $ is anti invariant iff $ \mathfrak{n} $ be anti invariant and also invariant. In both case $ \bar{J}A \in C^{1} $, thus $ C^{1} $ is invariant subalgebra of $ \mathfrak{\bar{g}} $. In the other case if $ \mathfrak{h} $ be invariant subalgebra, then $ A_{R} $ and $ A_{C} $ elements of $ C^{1} $. But if $ \mathfrak{h} $ be anti invariant subalgebra, then $ O(A_{C}) $ and $ O(A_{R}) $ are not zero, essentially.\\
Further $ A_{C} $ is in $ C^{2} $ and it is a cross element. But $ A_{R} $ is not in $ C^{2} $. It's trivial $ A_{C} $ is an element of $ C^{3} $. Ultimately, if A be in $ C^{3} $, then $ \bar{J}A \in C^{3} $, then $ C^{3} $ is invariant subalgebra and if $ \mathfrak{h} $ be anti invariant, then $ A_{R} $ is an element of $ C^{2} $.\\
Let M be a slant submanifold of G(You can get complete information about slant angle and slant submanifolds in \cite{ch}) contain identity element, then $ \mathfrak{m}=T_{e}M $
is a Lie subalgebra of $ \mathfrak{g} $ with slant angle $ \theta $, such that
\begin{align*}
\dfrac{g(JX,Y)}{\vert JX \vert \vert Y \vert}=\cos \theta
\end{align*}
$ X,Y \in \mathfrak{m} $ are arbitrary left invariant vector fields. Let $ A,B \in \bar{\mathfrak{m}} $, by definition we have
\begin{align*}
\langle \bar{J}A,B \rangle   &=\langle \bar{J}
\begin{bmatrix}
X & W \\
Z & Y  \\
\end{bmatrix},
\begin{bmatrix}
V & K \\
N & M  \\
\end{bmatrix}  \rangle ,\\
    &=g(JX,V)+g(JY,M)+g(JZ,N)+g(JW,S),\\
                                    &=\lbrace \vert JX\vert \vert V\vert +\vert JY\vert \vert M\vert + \vert JZ\vert \vert N\vert +\vert JW\vert \vert S \vert \rbrace \cos \theta.
\end{align*}
Thus slant angle of $ \mathfrak{m} $ and $ \bar{\mathfrak{m}} $ is equal. In all slant submanifolds we have $ JX=PX+FX $, that PX is tangent and FX is normal components of JX. Then 
\begin{align*}
\bar{J}A &=\bar{J}
\begin{bmatrix}
X & W \\
Z & Y  \\
\end{bmatrix} =\bar{J}
\begin{bmatrix}
PX+FX & PW+FW \\
PZ+FZ & PY+FY  \\
\end{bmatrix},\\
             &=\bar{P}
\begin{bmatrix}
X & W \\
Z & Y  \\
\end{bmatrix}+\bar{F}
\begin{bmatrix}
X & W \\
Z & Y  \\
\end{bmatrix}.
\end{align*}
Now we can see $ \bar{P} $ is tangent part and $ \bar{F} $ is normal part of $ \bar{J} $.
Reader could let J is abelian almost complex structure or $ \mathfrak{g} $ be a Kaehler Lie algebra and study matrix Lie groups. In the other case semi slant and the other cases of this type submanifolds could be studied by interested readers.
\subsection{2-step nilpotent Lie groups }
Let G be a simply connected, 2-step nilpotent Lie group equipped with a left invariant metric and $ \mathfrak{g} $ be the Lie algebra of N. We use the following decomposition 
\begin{align*}
\mathfrak{g}=Z(\mathfrak{g})+Z^{\perp}(\mathfrak{g}),
\end{align*}
$ Z(\mathfrak{g}) $ is the center of $ \mathfrak{g} $ and $ Z^{\perp}(\mathfrak{g}) $ is orthogonal complement of center with left invariant metric. We define the useful skew symmetric linear map 
$ j(Z):Z^{\perp}(\mathfrak{g})\rightarrow Z^{\perp}(\mathfrak{g}) $ by the equation 
$ g(j(Z)X,Y)=g([X,Y],Z) $, for all $ X,Y \in Z^{\perp}(\mathfrak{g}) $ and $ Z \in Z(\mathfrak{g}) $. N is of Heisenberg type if $ j^{2}(Z)=-\vert Z \vert^{2} Id $\cite{camb}. The Heisenberg Lie groups are excellent example of a contact manifold for our purpose. Using definition second section  we conclude for any $ A \in \mathfrak{\bar{g}} $, if all of entries are in $ Z(\mathfrak{\bar{g}}) $, then $ A \in Z(\mathfrak{\bar{g}}) $, otherwise if one of entries be in $ Z^{\perp}(\mathfrak{g}) $, A is an element of $ Z^{\perp}(\mathfrak{\bar{g}}) $. Using this concept and definition of bracket show that  $ \mathfrak{\bar{g}} $ is  2-step nilpotent if and only if $ \mathfrak{g} $ is 2-step nilpotent.
Suppose 
\begin{align*}
\mathbf{A}=
\begin{bmatrix}
X & Z^{*} \\
Z & Y  \\
\end{bmatrix},
\mathbf{B} = 
\begin{bmatrix}
Z & X \\
Y & Z^{*}  \\
\end{bmatrix} ,
\end{align*}
for all $ X,Y \in Z^{\perp}(\mathfrak{g}) $ and $ Z,Z^{*} \in Z(\mathfrak{g}) $. we will have the following facts. 
\begin{align*}
&(1)\bar{\nabla}_{A^{t}}A=\bar{\nabla}_{A}A^{t}=0,\\
&(2) \bar{\nabla}_{B^{t}}B=\dfrac{1}{2}
\begin{bmatrix}
0     & [Y,X] \\
[ X,Y] & 0  \\
\end{bmatrix}, \\
&(3) \bar{\nabla}_{B^{*}}B=0, \\
&(4) \bar{\nabla}_{A^{*}}A=\dfrac{1}{2}
\begin{bmatrix}
 [Y,X]  & 0  \\
 0 & [ X,Y] \\
\end{bmatrix}. 
\end{align*}
$ \bar{\nabla} $ is covariant derivative on $ \mathfrak{\bar{g}} $ . Using number (2) and (4) obtain
\begin{align*}
\langle \bar{\nabla}_{A^{*}}A,\bar{\nabla}_{B^{t}}B \rangle =0.
\end{align*}
It's trivial $ [A,B] \in Z(\mathfrak{\bar{g}}) $, thus $ A,B \in Z^{\perp}(\mathfrak{\bar{g}}) $.
Now we consider Heisenberg type Lie groups. Let G be a Heisenberg type Lie group and $ \bar{j} $ be a skew symmetric linear map on $ \mathfrak{\bar{g}} $ in the same way of j. For all 
\begin{align}
\mathbf{A} =
\begin{bmatrix}
X & M \\
N & Y  \\
\end{bmatrix} ,
\mathbf{B} =
\begin{bmatrix}
Z & W \\
V & S  \\
\end{bmatrix},
\end{align}
such that all entries of A are arbitrary elements of $ Z^{\perp}(\mathfrak{g}) $ and entries of B are from $ Z(\mathfrak{g}) $. Using definition of $ \nabla $ and $ \bar{\nabla} $, obtain 
\begin{align*}
\bar{\nabla}_{B}A &=\bar{j}(B)A=-\dfrac{1}{2}\bar{j}(
\begin{bmatrix}
Z & W \\
V & S  \\
\end{bmatrix}) 
\begin{bmatrix}
X & M \\
N & Y  \\
\end{bmatrix},\\
                         &=-\dfrac{1}{2}
\begin{bmatrix}
j(Z)X & j(W)M \\
j(V)N & j(S)Y  \\
\end{bmatrix}.
\end{align*}
Furthermore; by the following computation 
\begin{align*}
\bar{\nabla}_{B}\bar{\nabla}_{B}A &=
\begin{bmatrix}
\dfrac{1}{4}j(Z)^{2}X & \dfrac{1}{4}j(W)^{2}M \\
\dfrac{1}{4}j(V)^{2}N & \dfrac{1}{4}j(S)^{2}Y  \\
\end{bmatrix}\\
                                                  &=-\dfrac{1}{4}
\begin{bmatrix}
\mid Z\mid^{2}X & \mid W\mid^{2}M \\
\mid V\mid^{2}N & \mid S\mid^{2}Y  \\
\end{bmatrix}.                       
\end{align*}
Show that $ \bar{G} $ is not of Heisenberg type. Also, all Heisenberg Lie groups are contact metric manifold. Therefore dimension of it will be odd, but matrix Lie algebras in this paper are of even dimension. The other reason is that center of Heisenberg Lie groups are of dimension one, but the dimension of center of $ \bar{G} $ is four. Now we give some calculations.
 Suppose A and B are elements in (32) and 
\begin{align*}
\mathbf{C} =
\begin{bmatrix}
X^{'} & M^{'} \\
N^{'} & Y^{'}  \\
\end{bmatrix} 
,
\mathbf{B^{*}} =
\begin{bmatrix}
Z^{*} & W^{*} \\
V^{*} & S^{*}  \\
\end{bmatrix},
\end{align*}
Using (1.7) from \cite{camb} we obtain the following facts.
\begin{align*}
\langle \bar{j}(B)A,\bar{j}(B^{*})A  \rangle  &=g(Z,Z^{*})\mid X\mid^{2}+g(W,W^{*} )\mid M\mid^{2}\\
                                                                   &+g(V,V^{*})\mid N\mid^{2}
+g(S,S^{*})\mid Y\mid^{2}.
\end{align*}
\begin{align*}
\langle \bar{j}(B)A,\bar{j}(B)C \rangle  &=\mid Z\mid^{2}g(X,X^{'})+\mid W\mid^{2}g(M,M^{'})\\
                                                      &+\mid V\mid^{2}g(N,N^{'})+\mid W\mid^{2}g(Y,Y^{'}).
\end{align*}
\begin{align*}
\vert \bar{j}(B)A \vert =\vert X \vert \vert Z\vert +\vert M \vert \vert W\vert +\vert N\vert
\vert V \vert +\vert Y \vert \vert S \vert.
\end{align*}
\begin{align*}
 \bar{j}(B) \circ \bar{j}(B^{*}) +  \bar{j}(B^{*}) \circ \bar{j}(B)=-2 \lbrace g(Z,Z^{*})+g(W,W^{*})+g(V,V^{*})+g(S,S^{*}) \rbrace Id.
\end{align*}
\begin{lemma}
Let N be a 2-step nilpotent group of Heisenberg type. Then 
\begin{align*}
[X,j(Z)X]=\vert X \vert^{2}Z
\end{align*}
for all $ X \in Z^{\perp}(\mathfrak{g}) $ and $ Z \in Z(\mathfrak{g}) $.\cite{camb}
\end{lemma}
Suppose A,B are arbitrary elements of $ \bar{\mathfrak{g}} $ is presented in (10) such that $ B \in Z(\bar{\mathfrak{g}}) $, then

$[A,\bar{j}(B)A]= 
\begin{bmatrix}
[X,j(Z)X] & [M,j(W)M] \\
[N,j(V)N]  & [Y,j(S)Y]  
\end{bmatrix} =
\begin{bmatrix}
\vert X \vert^{2}Z & \vert M \vert^{2}W \\
\vert N \vert^{2}V & \vert Y\vert^{2}S 
\end{bmatrix}.$

If $ \vert X \vert^{2}=\vert M \vert^{2}=\vert N \vert^{2}=\vert Y\vert^{2}=1 $, then we will have $ [A,\bar{j}(B)A]=B $. \\
Let $ A,B \in C^{2} $, such that $ X,Y,N,M \in Z^{\perp}(\mathfrak{g}) $ and $ Z,W,V,S \in Z(\mathfrak{g}) $. In the first step we can see $ [A,B]=0 $. Also, we have $ \langle A,B \rangle =(A,B)=0 $ and $ \langle A,B^{t} \rangle =0 $. By the following calculations 
\begin{align*}
\bar{\nabla}_{B^{t}}A=-\dfrac{1}{2} 
\begin{bmatrix}
j(Z)X & j(V)M \\
j(W)N & j(S)Y   \\
\end{bmatrix},
\end{align*}
\begin{align*}
\bar{\nabla}_{A^{t}}B=-\dfrac{1}{2} 
\begin{bmatrix}
j(Z)X & j(W)N\\
j(V)M & j(S)Y   \\
\end{bmatrix}.
\end{align*}
We conclude 
\begin{align*}
(\bar{\nabla}_{B^{t}}A)^{t}=\bar{\nabla}_{A^{t}}B.
\end{align*}
Furthermore; we obtain 
\begin{align*}
[A^{*},A]=\dfrac{1}{2}\bar{\nabla}_{A^{*}}A.
\end{align*}

{\bf{Acknowledgements.}}
The author would welcome all valuble comments and suggestions.

\end{document}